\documentclass[11pt, a4paper]{article}
\usepackage{amsmath,amsthm,amscd,amsfonts,amssymb,enumitem}
\usepackage{latexsym}
\usepackage{multirow}

\usepackage[colorlinks=true, linkcolor=blue, anchorcolor=blue, citecolor=blue, filecolor=blue, urlcolor=blue]{hyperref}

\newtheorem{thm}{Theorem}[section]
\newtheorem{lemma}[thm]{Lemma}
\newtheorem{cor}[thm]{Corollary}

\theoremstyle{definition}
\newtheorem{defin}[thm]{Definition}
\newtheorem{assmp}{Assumption}
\newtheorem{exam}{Example}[section]
\theoremstyle{remark}
\newtheorem{rem}[thm]{Remark}

\newcommand{\qedwhite}{\hfill \ensuremath{\Box}}
\renewenvironment{proof}{{\raggedright \bfseries Proof.}}{\qedwhite}

\numberwithin{equation}{section}
\newcommand\blfootnote[1]{%
	\begingroup
	\renewcommand\thefootnote{}\footnote{#1}%
	\addtocounter{footnote}{-1}%
	\endgroup
}

\def\CC{\mathbb C}

\def\ZZ{\mathbb Z}
\def\RR{\mathbb R}

\def\PP{\mathbb P}

\def\oo{\mathcal O}
\def\aa{\mathcal A}
\def\nn{\mathcal N}
\def\ff{\mathcal F}
\def\suml{\sum\limits}
\def\supl{\sup\limits}

\title{Random sampling of signals concentrated on compact set in localized reproducing kernel subspace of $L^p(\RR^n)$}
\author{Dhiraj Patel\thanks{Email id: dpatel.iitd@gmail.com}, S. Sivananthan\thanks{Email id: siva@maths.iitd.ac.in}\\ 
Department of Mathematics, Indian Institute of Technology Delhi,\\ New Delhi-110016, India}

\date{}

\begin{document}
	\maketitle
	
	\begin{abstract}
		The paper is devoted to studying the stability of random sampling in a localized reproducing kernel space. We show that if the sampling set on $\Omega$ (compact) discretizes the integral norm of simple functions up to a given error, then the sampling set is stable for the set of functions concentrated on $\Omega$. Moreover, we prove with an overwhelming probability that $\oo(\mu(\Omega)(\log \mu(\Omega))^3)$ many random points uniformly distributed over $\Omega$ yield a stable set of sampling for functions concentrated on $\Omega$.
	\end{abstract}
	

	\section{Introduction}
	The sampling problem is one of the most active research area in the field of signal processing, image processing, and digital communication. The problem is related to find a discrete sample set such that a function $f$ can be uniquely determined and reconstructed by its discrete sample values. However, this problem is not well defined unless we assume some additional information on the function space. In this paper, we focus on the space of localized reproducing kernel subspace of $L^p(\RR^n):=L^p(\RR^n,\mu)$, where $\mu$ is a Lebesgue measure on $\RR^n$. \par 
	
	A countable set $X=\{ x_j\in \RR^n: j\in J \}$ is said to be a stable sample or stable set of sampling for the function space $V\subseteq L^p(\RR^n)$ if there exist positive constants $A$ and $B$ such that
	\begin{equation}
	\label{eqn:saminq}
	A\|f\|_{L^p(\RR^n)}^p\leq \sum_{j\in J} |f(x_j)|^p\leq B\|f\|_{L^p(\RR^n)}^p ~~\forall f\in V.
	\end{equation}
	In case of Paley-Wiener space $PW_{[a,b]}(\RR)$, the stability of a sampling set is completely characterized by Beurling density condition. However, a similar result is not valid for $PW_S(\RR^n)$, where $S$ is a convex subset of $\RR^n$, see \cite[Section 5.7]{olevskii2016functions}. Hence, to overcome these difficulties in studying non-uniform sampling in higher-dimension, we consider a set of random points and check the probability of a stable sampling set. At the same time, the problem of finding a stable random sampling on $\RR^n$ is not feasible in general. Bass and Gr\"{o}chenig \cite{bass2010random} observed that for each random sample identically and uniformly distributed over each cube $k + [0, 1]^n$ in $\RR^n$, the sampling inequality \eqref{eqn:saminq} fails almost surely for Paley-Wiener space. Moreover, for smooth function $f$ in $L^p(\RR^n)$, the sample value $f(x_j)$ may not assist in sampling inequality for large values of $x_j$. To resolve these problems, we consider random sample points are drawn uniformly and identically from a compact set $\Omega$, and a class of functions concentrated on $\Omega$. Such functions are useful in many application of engineering fields such as information and communication theory \cite{dilmaghani2003novel}, signal detection and estimation \cite{haykin1998signal}, neuroscience \cite{hoogenboom2006localizing}, optics \cite{frieden1971viii}, and many more.\par 
	
	Random sampling problem is closely related to learning theory \cite{cucker2002mathematical,poggio2003mathematics,smale2005shannon}, compressed sensing \cite{donoho2006compressed}, and widely applied in information recovery \cite{laska2006random}. Trigonometric polynomials are effectively used in practical applications such as computer tomography \cite{averbuch2001fast}, geophysics \cite{rauth1998smooth}, image processing \cite{strohmer1997computationally}, and cardiology \cite{strohmer1996recover}. Bass and Gr\"{o}chenig studied random sampling for multivariate trigonometric polynomial \cite{bass2005random}; Cand\'{e}s, Romberg, and Tao reconstructed sparse trigonometric polynomial from a random sample set \cite{candes2006robust}. In the last decades, random sampling studied for Paley-Wiener space \cite{bass2010random,bass2013relevant}; shift-invariant space \cite{fuhr2014relevant,yang2019random,yang2013random}; continuous function space with bounded derivative \cite{tao2019random}; function space with finite rate of innovation \cite{lu2019non}; reproducing kernel subspace of $L^p(\RR^n)$ which is an image of an idempotent integral operator \cite{li2020random,patel2020random}.\par 
	
	In this paper, we consider localized reproducing kernel subspace $V$ of $L^p(\RR^n)$ (defined in Section \ref{sec:preliminaries}), which takes into consideration of existing model spaces. A subset $V^*(\Omega,\delta)$ of $V$, is a set of \textit{$\delta$-concentrated} functions, define as follows $$V^*(\Omega,\delta):=\Big\{ f\in V: \int_{\Omega} |f(x)|^p\, dx\geq (1-\delta)\|f\|_{L^p(\RR^n)}^p \Big\}, ~~0<\delta<1.$$ We are interested in finding probability bound for random sample $\{\xi_\nu: \nu=1,\dots,r \}$ uniformly and identically drawn from $\Omega$ to be a stable sample set for $V^*(\Omega,\delta)$ and satisfy the sampling inequality
	\begin{equation}
	\label{eqn:ransaminq}
	A\|f\|_{L^p(\RR^n)}^p\leq \frac{1}{r}\sum_{\nu=1}^{r}|f(\xi_\nu)|^p\leq B\|f\|_{L^p(\RR^n)}^p ~~\forall f\in V^*(\Omega,\delta).
	\end{equation}
	Of course, the sampling inequality can be achievable with high probability for large-scale sample size. However, the problem seems interesting if one can find a minimal the sample size required to satisfy the sampling inequality \eqref{eqn:ransaminq}. It was proved for $\delta$-concentrated functions on the cube $C_R:=[ -R/2,R/2 ]^n$ that sample size required to be of $\oo(R^{2n})$ in case of Paley-Wiener space \cite{bass2010random}, shift-invariant space \cite{yang2019random}, and image of an idempotent integral operator \cite{patel2020random}. However, in recent years, it was shown that effective number of sample is indeed of order $\oo(R^n\log R^n)$ for Paley-Wiener space \cite{bass2013relevant} and shift-invariant subspace of $L^2(\RR^n)$ \cite{fuhr2014relevant}. The recent article \cite{li2020random} by Li et al. proved that for the space of image of an idempotent integral operator, $\delta$-concentrated functions on Corkscrew domain $\Omega$ satisfy \eqref{eqn:ransaminq} with sample size of order $\oo(\nu(\Omega)\log\nu(\Omega))$, where $\nu$ denotes the metric measure on $\Omega$.
	
	The main features of this paper are summarized as follow:
	\begin{enumerate}[label=(\roman*)]
		\item In the recent articles \cite{li2020random,patel2020random}, the random sampling problem studied for an image of an idempotent integral operator $T$, with an additional assumption that the integral kernel $K$ satisfies the regularity condition $$\lim\limits_{\epsilon\rightarrow 0}\left\| \sup\limits_{z\in \RR^n}|w_{\epsilon}(K)(\cdot+z,z)|\right\|_{L^1(\RR^n)}=0,$$ where $w_{\epsilon}(K)(x,y)=\supl_{x',y'\in [-\epsilon,\epsilon]^n}|K(x+x',y+y')-K(x,y)|.$ In this paper, we drop this assumption and study sampling inequality \eqref{eqn:ransaminq} for localizable reproducing kernel space. Further, instead of considering signals concentrated on cube $C_R=[ -R/2,R/2 ]^n$ in $n$-dimensional Euclidean space, we study the random sampling problem for signals concentrated on a compact subset $\Omega$ of $\RR^n$.
		
		\item  We show that any element in $V$ can be approximated by an element in a finite-dimensional subspace of $V$. As a consequence, we can prove the random sampling inequality \eqref{eqn:ransaminq} using the same line of proof in \cite{bass2010random,patel2020random,yang2013random}. However, it does not lead us to better sample size estimation. In this paper, we apply the idea of Marcinkiewicz type discretization result introduced in \cite{dai2020sampling} to solve random sampling problem in localizable reproducing kernel space. We show that if a sampling set is ``good'' discretization to the integral norm on $\Omega$ for the class of simple functions, then it is a stable sampling set for $\delta$-concentrated functions on $\Omega$. In addition, we prove that the sampling inequality \eqref{eqn:ransaminq} can be achievable with high probability if the sample size of order $\oo(\mu(\Omega)(\log \mu(\Omega))^3)$.
	\end{enumerate}
	
	We pursue the approach of \cite{dai2020sampling,yang2019random} with a mild condition on generators. Note that the result in \cite{yang2019random} based on strong decay condition of generators $$\phi(x)\leq \frac{C}{(1+|x|)^m}, ~~x\in \RR^n,$$ and in \cite{dai2020sampling} relied on the assumption of boundedness of entropy number.\par 
	
	This paper is organized as follows. In Section \ref{sec:preliminaries}, we give basic definitions, notations and preliminary results. In Section \ref{sec:coveringnumber}, we show that functions in a given compact set are bounded by simple functions. Moreover, under some condition on simple functions, we show that the sampling inequality hold for functions concentrated in a compact set. The main result of this paper is provided in Section \ref{sec:randomsampling}.
	
	\section{Preliminaries}
	\label{sec:preliminaries}
	
	In this section, we define the localized reproducing kernel subspace $V$ of $L^p(\RR^n),$ and discuss some interesting examples. 
	
	A sample set $U=\{ u : u\in \RR^n \}$ is \textit{relatively separated} with positive \textit{gap} $\beta$ if $$\beta=\inf_{\underset{u\neq u'}{u,u'\in U}} \| u-u' \|_{\infty}>0.$$
	
	The \textit{Wiener amalgam space} $W(L^1)(\RR^n)$ consist of all functions $f\in L^{\infty}(\RR^n)$ such that $$\|f\|_{W(L^1)(\RR^n)}:=\sum_{k\in \ZZ^n} \sup_{x\in [0,1]^n} |f(x+k)|<\infty.$$
	
	A family $\{ x_i:i\in I \}$ of elements of a Hilbert space $H$ is called a \textit{frame} for $H$ if there exist constants $c_1,c_2>0$ such that $$c_1\|x\|^2\leq \sum_{i\in I} |\langle x,x_i \rangle|\leq c_2\|x\|^2 ~~~\forall x\in H.$$ 
	\begin{defin}
		We say that a closed subspace $W\subseteq L^2(\RR^n)$ is \textit{localizable reproducing kernel Hilbert space}, if there exist: 
		\begin{enumerate}[label=(\alph*)]
			\item a relatively separated set $\Gamma\subseteq \RR^n$ (nodes);
			
			\item a function $\Theta\in W(L^1)(\RR^n)$ (envelope);
			
			\item a collection of continuous functions $\{ F_{\gamma}:\gamma\in \Gamma \}$  is a frame for $W,$ and satisfy the localization estimate
			\begin{equation}
			\label{eqn:framebound}
			|F_{\gamma}(x)|\leq \Theta(x-\gamma), ~~~\forall\gamma\in \Gamma.
			\end{equation}
		\end{enumerate} 
	\end{defin}
	The coefficient map $f\mapsto Cf:=(\langle f,F_{\gamma} \rangle)_{\gamma\in \Gamma}$ is bounded from $L^2(\RR^n)\rightarrow \ell^2(\Gamma)$, and can be extended to the bounded operator $C:L^p(\RR^n)\rightarrow \ell^p(\Gamma)$. Likewise, the adjoint operator $C^*:\ell^p(\Gamma)\rightarrow L^p(\RR^n), ~~1\leq p<\infty$ defined by $c\mapsto C^*c:=\sum_{\gamma\in \Gamma} c_{\gamma}F_{\gamma}$ is bounded. Using the collection $\{ F_{\gamma}:\gamma\in \Gamma \}$ is a frame, we get the range space $$V:=C^*(\ell^p(\Gamma))=\Big\{ \sum_{\gamma\in \Gamma} c_{\gamma}F_{\gamma}: c\in \ell^p(\Gamma) \Big\}$$ is a well-defined closed subspace of $L^p(\RR^n)$, for $1\leq p<\infty,$ and there exist constants $A,B>0$ such that
	\begin{equation}
	\label{eqn:frameinq}
	A\|f\|_{L^p(\RR^n)}^p\leq \sum_{\gamma\in \Gamma} |c_{\gamma}|^p\leq B\|f\|_{L^p(\RR^n)}^p,
	\end{equation}
	for every $f=\suml_{\gamma\in \Gamma} c_{\gamma}F_{\gamma} \in V$, (see \cite{grochenig2004localization}). Moreover, it is easy to show that $V$ is a reproducing kernel Banach space. First, let us recall a basic definition.\\
	
	\begin{defin}
		A Banach space $\Sigma$ of functions on a set $X$ is a \textit{reproducing kernel Banach space} if the point evaluation functional $f \mapsto f(x)$ is continuous for each $x\in X$, i.e., for every $x\in X$, there exists $C_x>0,$ such that $|f(x)|\leq C_x \|f\|,$ for all $f\in \Sigma.$
	\end{defin}

	\begin{lemma}
		The space $V$ is a reproducing kernel Banach space.
	\end{lemma}

	\begin{proof}
		Let $x\in \RR^n$ be fixed and $f\in V$ be arbitrary.
		\begin{align*}
		f(x)&=\suml_{\gamma\in \Gamma} c_{\gamma}F_{\gamma}(x),\\
		|f(x)|&\leq \Big( \sum_{\gamma\in \Gamma} |c_{\gamma}|^p \Big)^{\frac{1}{p}}\Big( \sum_{\gamma\in \Gamma} |F_{\gamma}(x)|^{p'} \Big)^{\frac{1}{p'}}\\
		&\leq B^{\frac{1}{p}}\|f\|_{L^p(\RR^n)}\Big( \sum_{\gamma\in \Gamma} |\Theta(x-\gamma)|^{p'} \Big)^{\frac{1}{p'}}\\
		&\leq B^{\frac{1}{p}}\|f\|_{L^p(\RR^n)}\Big( \sum_{\gamma\in \Gamma} |\Theta(x-\gamma)| \Big)^{\frac{1}{p'}}\\
		&=B^{\frac{1}{p}}\|f\|_{L^p(\RR^n)}\Big( \sum_{k\in \ZZ^n} \sum_{\gamma\in \Gamma\cap [k,k+1]^n} |\Theta(x-\gamma)| \Big)^{\frac{1}{p'}}\\
		&\leq B^{\frac{1}{p}}\|f\|_{L^p(\RR^n)}\Big( \sum_{k\in \ZZ^n} N(\Gamma) \sup_{x\in [0,1]^n}|\Theta(x-k)| \Big)^{\frac{1}{p'}}\\
		|f(x)|&\leq B^{\frac{1}{p}}N(\Gamma)^{\frac{1}{p'}}\|\Theta\|_{W(L^1)(\RR^n)}^{\frac{1}{p'}}\|f\|_{L^p(\RR^n)}.
		\end{align*}
		Therefore, the point evaluation functional is continuous for each $x\in \RR^n$.
	\end{proof}

	From the proof of the above lemma, we choose a constant $D>1$ such that
	\begin{equation}
	\label{eqn:reproducinginq}
	\|f\|_{L^{\infty}(\RR^n)}\leq D\|f\|_{L^p(\RR^n)}, ~~~\forall f\in V.
	\end{equation}
	
	\begin{exam}
		In the following, we give some well known examples of localizable reproducing kernel spaces.
		\begin{enumerate}
			\item Let $\phi\in W(L^1)(\RR^n)$ be such that for some positive constants $A_1$ and $B_1$, $$A_1\leq \sum_{k\in \ZZ^n} |\hat{\phi}(\xi+k)|^2\leq B_1.$$ Then $\{ \phi(\cdot-k): k\in \ZZ^n \}$ is a frame for the shift-invariant space $V(\phi)=\Big\{ \sum\limits_{k\in \ZZ^n} c_k\phi(\cdot-k): (c_k)\in \ell^2(\ZZ^n) \Big\}$, see \cite[Theorem 9.2.5]{christensen2003introduction}. In this case, we have $F_{\gamma}=\phi(\cdot-\gamma),\, \gamma\in\Gamma=\ZZ^n$.
			
			\item Let $T$ be an idempotent integral operator ($T^2=T$) on $L^p(\RR^n)$ with the integral kernel $K$ satisfies the off-diagonal decay condition $$\|K\|_{S}:=\max\Big( \sup_{x\in \RR^n}\|K(x,\cdot)\|_{L^1(\RR^n)},\sup_{y\in \RR^n}\|K(\cdot,y)\|_{L^1(\RR^n)} \Big)<\infty,$$  and the regularity condition
			\begin{equation}
			\label{eqn:regularity}
			\lim_{\epsilon\rightarrow 0} \|w_{\epsilon}(K)\|_{S}=0,
			\end{equation}
			where $w_{\epsilon}(K)(x,y)=\sup\limits_{x',y'\in [-\epsilon,\epsilon]^n} |K(x+x',y+y')-K(x,y)|.$ Then the image space $V$ of $T$ is a reproducing kernel subspace of $L^p(\RR^n),$ and there exist relatively separated set $\Gamma\subseteq \RR^n$ and the collection $\{ \phi_{\gamma}: \gamma\in \Gamma \}$ forms $p$-frame for $V$. Apart from that there exists $h\in W(L^1)(\RR^n)$ such that $|\phi_{\gamma}(x)|\leq h(x-\gamma),$ see \cite{nashed2010sampling}.
		\end{enumerate}
	\end{exam}
	
	So far, the sampling problem had been studied for the image of an idempotent integral operator in \cite{nashed2010sampling,li2020random,patel2020random} and the regularity condition \eqref{eqn:regularity} was the key assumption. However, in practice the condition is not feasible. In the following, we give examples of localizable reproducing kernel space, which is the image of an idempotent integral operator but the regularity condition is either hard to verify or does not satisfy.
	
	\begin{exam}
		\label{exam:quasi-shift}
		Let $X=\{ x_k:k\in \ZZ \}$ be a relatively separated set of strictly increasing sequence in $\RR$ and $\phi \in W(L^1)(\RR)$. A quasi-shift invariant space is defined by $$V_X(\phi)=\Big\{ f=\sum_{k\in \ZZ} c_k\phi(\cdot-x_k) : c=(c_k)\in \ell^2(\ZZ) \Big\}.$$ For more literature on quasi-shift invariant space, we refer the interested reader to \cite{feichtinger2008perturbation,kumar2020sampling}. If the collection $\{ \phi(\cdot-x_k): k\in \ZZ \}$ is a frame for $V_X(\phi),$ then the space $V_X(\phi)$ is a localizable reproducing kernel Hilbert space. Furthermore, $V_X(\phi)$ can be written as the image of an idempotent integral operator with the integral kernel $$K(x,y)=\sum_{k\in \ZZ} \overline{\phi(x-x_k)}\tilde{\phi_k}(y),$$ where $\{ \tilde{\phi_k}:k\in \ZZ \}$ is the canonical dual frame of $\{ \phi(\cdot-x_k): k\in \ZZ \}$. In case of irregular sample set $X$, there is no constructive method to compute $\tilde{\phi_k}.$ As a consequence, the condition on $K$ is not easily verifiable. Even for the simplest case, the kernel assumption \eqref{eqn:regularity} was not satisfied. For example, let $\phi:\RR\rightarrow \RR$ be defined by
		\begin{align*}
		\phi(x)=\begin{cases}
		1,  & \text{if } 0\leq x\leq \frac{1}{2}, \\
		0, & \text{otherwise}.
		\end{cases}
		\end{align*}
		Then $\phi\in W(L^1)(\RR)$ and $\|\phi\|_{W(L^1)(\RR)}=1.$ The collection $\{ \sqrt{2}\phi(\cdot-k):k\in \ZZ \}$ is an orthonormal basis of the shift-invariant space $V(\phi)\subset L^2(\RR)$ and the reproducing kernel of $V(\phi)$ is defined by $K(x,y)=\sum\limits_{k\in \ZZ} 2\phi(x-k)\phi(y-k).$\par 
		For $\epsilon<\frac{1}{2}$,
		\begin{align*}
		\sup_{x'\in [-\epsilon,\epsilon]} |K(x+x',y)-K(x,y)|=\begin{cases}
		2\phi(y-k), & \text{if } x\in (k-\epsilon,k+\epsilon)\cup(k+\frac{1}{2}-\epsilon,k+\frac{1}{2}+\epsilon),\\
		0, & \text{otherwise}.
		\end{cases}
		\end{align*}
		Hence,
		\begin{align*}
		\Big\|\sup_{x'\in [-\epsilon,\epsilon]} |K(x+x',\cdot)-K(x,\cdot)|\Big\|_{L^1(\RR)}=\begin{cases}
		1, &\text{if } x\in (k-\epsilon,k+\epsilon)\cup(k+\frac{1}{2}-\epsilon,k+\frac{1}{2}+\epsilon),\\
		0, &\text{otherwise}.
		\end{cases}
		\end{align*}
		Therefore,
		\begin{align*}
		1&=\sup_{x\in \RR} \Big\|\sup_{x'\in [-\epsilon,\epsilon]} |K(x+x',\cdot)-K(x,\cdot)|\Big\|_{L^1(\RR)}\\
		&\leq \max\left( \sup_{x\in \RR}\|w_{\epsilon}(K)(x,\cdot)\|_{L^1(\RR)},\sup_{y\in \RR}\|w_{\epsilon}(K)(\cdot,y)\|_{L^1(\RR)} \right)=\|w_{\epsilon}(K)\|_{S}.
		\end{align*}
		
	\end{exam}
	
	We collect our assumption on compact domain $\Omega$ and the function $\Theta$ in the following list:
	
	\begin{assmp}
		Without loss of generality assume that $\mu(\Omega)\geq 1,$ and denote $d$ as the number of unit cube of the form $m+[0,1]^n$ covers the boundary of $\Omega$, where $m\in \ZZ^n$. Let the number of $\Gamma$ in $\Omega$ is bounded by $C(\Gamma)\mu(\Omega),$ where $C(\Gamma)$ is some positive constant.
	\end{assmp}
	
	\begin{assmp}
		\label{assmp:Theta}
		For every $C_N=[-N/2,N/2]^n$,
		\begin{equation}
		\label{eqn:thetaconver}
		\sum_{k\in \ZZ^n\setminus C_N} \sup_{x\in [0,1]^n} |\Theta(x-k)|< \frac{C}{N^{n\alpha}},
		\end{equation}
		where $C$ are positive constant, and $\alpha\geq\frac{p}{p-1}$ for $1<p<\infty$ and $\alpha\geq 1$ for $p=1$.
	\end{assmp}

	In the following, we provide an example of localizable reproducing kernel space satisfying Assumption \ref{assmp:Theta}.
	\begin{exam}
		Let $\phi:\CC^n\rightarrow \RR$ be a plurisubharmonic function and assume that there exist $m,M>0$ such that
		\begin{equation}
		\label{eqn:plurisub}
		im\partial\bar{\partial}|z|^2\leq i\partial\bar{\partial}\phi\leq iM\partial\bar{\partial}|z|^2.
		\end{equation}
		Let $A_{\phi}^2$ be the space of entire function on $\CC^n$ equipped with the norm $$\|f\|_{\phi,2}^2:=\int_{\CC^n} |f(z)|^2e^{-2\phi(z)}\,dz.$$ Then $A_{\phi}^2$ is a reproducing kernel Hilbert space and the reproducing kernel is denoted by $K_{\phi}(z,w).$ Moreover, the kernel $K_{\phi}$ satisfy the following off-diagonal decay estimate \cite{delin1998pointwise} $$|K_{\phi}(z,w)|e^{-\phi(z)-\phi(w)}\leq Ce^{-c|z-w|} ~~\forall z,w\in \CC^n.$$ We define the weighted Fock space $A_{\phi}^2$ as $$V_{\phi}^2=\{ f=ge^{-\phi} : g\in A_{\phi}^2 \}.$$ The space $V_{\phi}^2$ is a localizable reproducing kernel Hilbert space in $L^2(\RR^{2n})$, see \cite{grochenig2019strict}. If $\beta\in (0, \sqrt{2/n})$, then $\Gamma=\beta\ZZ^{2n}$ is a relatively separated set and there exists $\{ F_{\gamma}: \gamma\in \Gamma \}$ frame for $V_{\phi}^2$ with frame bounds $A=\frac{1}{4\beta^{2n}}$ and $B=\frac{3}{2}.$ In addition, the frame $\{ F_{\gamma}: \gamma\in \Gamma \}$ satisfy the localized estimate \eqref{eqn:framebound}, where $\Theta(z)=C_{\beta}e^{-c|z|}$. It is easy to verify that $\Theta\in W(L^1)(\RR^{2n})$ and satisfy Assumption \ref{assmp:Theta}.
	\end{exam}
	
	In the following lemma, we show that for any $f$ in $V$ there exists a function in finite-dimensional subspace of $V$ which is close to $f$. In order to proceed the lemma, let $M$ be a compact set and we define finite-dimensional subspace $V_M$ of $V$ as $$V_M=\Big\{ \suml_{\gamma\in \Gamma\cap M} c_{\gamma}F_{\gamma} : c_{\gamma}\in \RR \Big\}.$$
	
	\begin{lemma}
		\label{lemma:p-approx}
		For a given $\epsilon>0$ and $f\in V$, there exist compact set $M$ (depending on $\epsilon$ and $\Omega$) and $\tilde{f}\in V_M$ such that
		\begin{equation*}
		\|f-\tilde{f}\|_{L^p(\Omega)}<\epsilon\|f\|_{L^p(\RR^n)}.
		\end{equation*}
	\end{lemma}
	
	\begin{proof}
		Let $f\in V$. Then there exists $(c_{\gamma})\in \ell^p(\Gamma)$ such that $f=\suml_{\gamma\in \Gamma} c_{\gamma}F_{\gamma}$. We consider $\tilde{f}=\suml_{\gamma\in \Gamma\cap M} c_{\gamma}F_{\gamma}$, where $M$ is a compact subset of $\RR^n$ (chosen later).
		
		Now, for $x\in \Omega$,
		\begin{align*}
		|f(x)-\tilde{f}(x)|&\leq \suml_{\gamma\in \Gamma\setminus M} |c_{\gamma}||F_{\gamma}(x)|\\
		&\leq \Big(\sum_{\gamma\in \Gamma\setminus M} |c_{\gamma}|^p \Big)^{\frac{1}{p}}\Big( \sum_{\gamma\in \Gamma\setminus C_N} |F_{\gamma}(x)|^{p'} \Big)^{\frac{1}{p'}}\\
		&\leq B^{\frac{1}{p}}\|f\|_{L^p(\RR^n)}\Big( \sum_{\gamma\in \Gamma\setminus M} |\Theta(x-\gamma)|^{p'} \Big)^{\frac{1}{p'}}\\
		&\leq B^{\frac{1}{p}}\Big( \sum_{\gamma\in \Gamma\setminus M} |\Theta(x-\gamma)| \Big)^{\frac{1}{p'}}\|f\|_{L^p(\RR^n)},
		\end{align*}
		Since $\Theta\in W(L^1)(\RR^n)$, for each $\epsilon>0$ there exists $N=\Big( \frac{CN(\Gamma)B^{\frac{p'}{p}}}{\epsilon^{p'}} \Big)^\frac{1}{n\alpha}\mu(\Omega)^{\frac{1}{n}}$ such that $$\sum_{k\in \ZZ^n\setminus C_{N}} \sup_{x\in [0,1]^n}|\Theta(x-k)|\leq \frac{C}{N^{n\alpha}}<\frac{\epsilon^{p'}}{N(\Gamma)\mu(\Omega)^{p'}B^{\frac{p'}{p}}}.$$ For $x\in \partial\Omega$, there exists $m\in \ZZ^n$ such that $x\in m+[0,1]^n$. Consider
		\begin{align*}
		\sum_{\gamma\in \Gamma\setminus \{ m+C_N \}} |\Theta(x-\gamma)|&\leq N(\Gamma)\sum_{k\in \ZZ^n\setminus C_N} \sup_{y\in [0,1]^n} |\Theta(y-k)|\\
		&<\frac{\epsilon^{p'}}{\mu(\Omega)^{p'}B^{\frac{p'}{p}}}.
		\end{align*}
		The same is true for every $x\in \partial\Omega\cap \{ m_i+[0,1]^n \},$ where $m_i\in \ZZ^n$ and the collection $\{ m_i+[0,1]^n:i=1,\dots,d\}$ cover $\partial \Omega.$ Choose $M= \Big( \cup_{i=1}^d \{m_i+C_N\} \Big)\cup \Omega.$ Then for each $x\in \Omega$, $$\sum_{\gamma\in \Gamma\setminus M} |\Theta(x-\gamma)|<\frac{\epsilon^{p'}}{\mu(\Omega)^{p'}B^{\frac{p'}{p}}},$$ and hence
		\begin{equation}
		\label{eqn:infty-approx}
		|f(x)-\tilde{f}(x)|<\frac{\epsilon}{\mu(\Omega)}\|f\|_{L^p(\RR^n)}, ~~~\forall x\in \Omega.
		\end{equation}
		This completes the proof.		
	\end{proof}
	
	\begin{rem}
		The space $V_M$ is generated by $\{ F_{\gamma}:\gamma\in \Gamma\cap M \}$ and the dimension $d_M$ of $V_M$ is at most the number of $\gamma$ in $M$. Therefore,
		\begin{align*}
		d_M\leq | \Gamma\cap M |&\leq \Big( C(\Gamma)\mu(\Omega)+dN(\Gamma)N^n \Big)\\
		&\leq \mu(\Omega)\Bigg( dN(\Gamma)\Big( \frac{CN(\Gamma)B^{\frac{p'}{p}}}{\epsilon^{p'}} \Big)^\frac{1}{\alpha}+C(\Gamma) \Bigg)\\
		&=C(\epsilon)\mu(\Omega),
		\end{align*}
		where $C(\epsilon)=\Bigg( dN(\Gamma)\Big( \frac{CN(\Gamma)B^{\frac{p'}{p}}}{\epsilon^{p'}} \Big)^\frac{1}{\alpha}+C(\Gamma) \Bigg)$ and $N(\Gamma)=\supl_{k\in \ZZ^n} |\Gamma\cap(k+[0,1]^n)|.$ 
	\end{rem}
	
	\begin{rem}
		
In the proof of the above Lemma, if we do not choose $M$ carefully, the bound of $d_M$ may not be effective estimation. For example, if $d$ is the number of unit cubes that cover $\Omega$, then $d>\mu(\Omega)$. On the other hand, if we consider a cube $Q$ in $\RR^n$ which contains $\Omega$, then $\mu(Q)\geq \mu(\Omega).$ In both cases, we get a larger bound for $d_M$ as $C\mu(\Omega)^2$ for some $C>0$.

	\end{rem}

	\section{Discretization of functions}
	\label{sec:coveringnumber}
	
	Let $(X,\|\cdot\|)$ be a Banach space and $B(f,r)$ denotes a ball of radius $r$ center at $f$. For a compact set $A$ and a positive number $\epsilon$, we define the covering number $N_{\epsilon}(A)$ as follows $$N_{\epsilon}(A):=N_{\epsilon}(A,X):=\min\Big\{ k: \exists f_1,f_2,\dots,f_k\in A,\, A\subseteq \bigcup\limits_{j=1}^k B(f_j,\epsilon) \Big\}.$$ The corresponding minimal $\epsilon$-net is denoted by $\nn_{\epsilon}(A,X)$, and $N_{\epsilon}(A,X)=|\nn_{\epsilon}(A,X)|$. The following lemma is a well-known estimation of covering number of a closed ball in finite-dimensional space. 
	
	\begin{lemma}[\cite{cucker2007learning}]
		\label{lemma:finitedimcoveringno}
		Let $X$ be a Banach space of dimension $s$. Then the number of open balls of radius $\omega$ to cover $\overline{B(0;r)}$ is bounded by $\left( \frac{2r}{\omega}+1 \right)^s$.
	\end{lemma}

	Let $M$ be a compact subset of $\RR^n$. We consider the set $$V_{M,\Omega}=\Big\{ f\in V_M: \|f\|_{L^p(\RR^n)}^p=\mu(\Omega) \Big\}.$$ From \eqref{eqn:reproducinginq}, we see that $V_{M,\Omega}$ is a compact subset of $L^\infty(\RR^n)$ and bounded by $D\mu(\Omega)^{\frac{1}{p}}.$ Lemma \ref{lemma:finitedimcoveringno} gives a bound of covering number of $V_{M,\Omega}\subseteq L^\infty(\RR^n)$, i.e.
	\begin{equation}
	\label{eqn:coveringbound}
	N_{\epsilon}(V_{M,\Omega},L^\infty(\RR^n))\leq \Big( 1+\frac{2D\mu(\Omega)^{\frac{1}{p}}}{\epsilon} \Big)^{d_M}\leq \Big( \frac{4D\mu(\Omega)^{\frac{1}{p}}}{\epsilon} \Big)^{d_M}.
	\end{equation}
	
	In the following, we now discuss on discretization of functions in a compact set $V_{M,\Omega}$. Let $a\in (0,\frac{1}{2}]$ be a fixed small number (chosen later) and denote the set $\aa_j=\nn_{a(1+a)^j}(V_{M,\Omega},L^\infty(\RR^n))$ for $j\in \ZZ$. Let $j_0\in \ZZ$ be a fixed integer (specified later), and for each $j(\geq j_0)\in \ZZ,$ consider the map $A_j:V_{M,\Omega}\rightarrow \aa_j$ such that $A_j(f)$ is a function in $\aa_j$ closest to $f$ with respect to $\|\cdot\|_{L^\infty(\RR^n)}.$ Then
	\begin{equation}
	\label{eqn:Aj-approx}
	\|f-A_j(f)\|_{L^\infty(\RR^n)}\leq a(1+a)^j.
	\end{equation}
	If $f\in V_{M,\Omega}$ and $j\in \ZZ\cap (j_0,\infty)$, we construct a set of collection of points as follow.
	\begin{gather*}
	U_j(f):=\Big\{ x\in \RR^n:|A_j(f)(x)|\geq (1+a)^{j} \Big\},\\
	D_j(f):=U_j(f)\setminus \cup_{k>j} U_k(f), ~~~~~ D_{j_0}(f):=\RR^n\setminus \cup_{k>j_0} U_k(f).
	\end{gather*}
	For each $f\in V_{M,\Omega}$, we define piecewise constant function $h(f)$ by $$h(f)=\sum_{j>j_0} (1+a)^j\chi_{D_j(f)},$$ where $\chi_E$ is the characteristic function on $E$.\\
	
	The following lemma gives some properties of the above $h$-mapping. We show that the absolute value of the function $f$ in $V_{M,\Omega}$ is bounded above and below by constants times of $h(f)$. 
	\begin{lemma}
		For each $f\in V_{M,\Omega}$, 
		\begin{equation}
		\label{eqn:fdiscreteinq}
		C_1(a)h(f)(x)\leq |f(x)|\leq C_2(a)h(f)(x) ~~\forall x\in \RR^n\setminus D_{j_0}(f),
		\end{equation}
		and
		\begin{equation}
		\label{eqn:fboundinDj0}
		|f(x)|\leq C_2(a)(1+a)^{j_0} ~~\forall x\in D_{j_0}(f),
		\end{equation}
		where $C_1(a)=(1-a)$ and $C_2(a)=(1+a)^2$.
	\end{lemma}
	\begin{proof}
		Let $x\in D_j(f)$ with $j>j_0$. Then $x\in U_j(f)$ and $x\notin U_k(f)$ for $k>j$. Using \eqref{eqn:Aj-approx} and from the definition of $U_j(f)$, we get
		\begin{equation*}
		|f(x)|\geq |A_j(f)(x)|-a(1+a)^j\geq (1+a)^{j}-a(1+a)^j=C_1(a)(1+a)^j
		\end{equation*}
		and 
		\begin{equation*}
		|f(x)|\leq |A_{j+1}(f)(x)|+a(1+a)^{j+1}\leq (1+a)^{j+1}+a(1+a)^{j+1}=C_2(a)(1+a)^j,
		\end{equation*}
		where $C_1(a)=(1-a)$ and $C_2(a)=(1+a)^2.$ Therefore, for all $x\in \RR^n\setminus D_{j_0}(f)=\cup_{j>j_0} D_j(f)$ we have,
		\begin{equation*}
		C_1(a)h(f)(x)\leq |f(x)|\leq C_2(a)h(f)(x).
		\end{equation*}
		Similarly, the other inequality \eqref{eqn:fboundinDj0} can be derived.
	\end{proof}
	
	\begin{rem}
		For all $a\in (0,\frac{1}{2}],$ $C_1(a)\leq 1\leq C_2(a),$ and $$\lim_{a\rightarrow 0} C_1(a)=\lim_{a\rightarrow 0} C_2(a)=1.$$ For our results, we choose $a\in (0,\frac{1}{2}]$ such that $\Big( \frac{C_2(a)}{C_1(a)} \Big)^p\leq \frac{5}{4}.$
	\end{rem}

	We list some notations which will be used in the rest of the following sections.\\
	
	\begin{tabular}{ p{6cm} p{8cm} }
		\hline
		Symbol & Remark \\
		\hline
				$|X|=$ The number of element in $X$ & $X$ is a finite set.\\
		$N(\Gamma)=\supl_{k\in \ZZ^n} |\Gamma\cap (k+[0,1]^n)|$ & Intuitively, the maximum number of elements of $\Gamma$ in a unit cube of $\RR^n$. \\
		$p'=\frac{p}{p-1}$ & $p'=\infty$ for $p=1$.\\
		$C_1(a)=(1-a)$ & \multirow{2}{8cm}{with $a$ satisfying $\Big( \frac{C_2(a)}{C_1(a)} \Big)^p\leq \frac{5}{4}$.}\\
		$C_2(a)=(1+a)^2$ & \\
		\hline
	\end{tabular}
	\vspace{0.5cm}
	
	In the following lemma, we give a condition on $h$-mapping such that for functions in $V_{M,\Omega}$ the sample set discretize the integral norm on $\Omega$. For a sample set $\xi=\{ \xi_\nu \}_{\nu=1}^r$ and a function $f\in V$, denote $$S(f,\xi):=(f(\xi_1),\dots,f(\xi_r))\in \RR^r, ~~~ \|S(f,\xi)\|_p^p:=\frac{1}{r}\sum_{\nu=1}^r |f(\xi_\nu)|^p.$$
	
	\begin{lemma}
		\label{lemma:saminqcon}
		Let  $\xi$ be a sample set in $\Omega$ and $f\in V_{M,\Omega}$. Assume that the function $h(f)$ satisfy the following inequality 
\begin{equation}\label{eqn:assumeh}
\frac{1}{\mu(\Omega)}\|h(f)\|_{L^p(\Omega)}^p-\sigma\leq \|S(h(f),\xi)\|_p^p\leq \frac{1}{\mu(\Omega)}\|h(f)\|_{L^p(\Omega)}^p+\sigma,
\end{equation}		
		for some $\sigma>0$. Then
		\begin{multline}
		\label{eqn:saminqVNp}
		C_1(a)^p\Big(\frac{C_2(a)^{-p}}{\mu(\Omega)}\|f\|_{L^p(\Omega)}^p-(1+a)^{pj_0}-\sigma\Big)\\
		\leq \|S(f,\xi)\|_p^p\leq C_2(a)^p\Big(\frac{C_1(a)^{-p}}{\mu(\Omega)}\|f\|_{L^p(\Omega)}^p+(1+a)^{pj_0}+\sigma\Big).
		\end{multline}
	\end{lemma}
	
	\begin{proof}
		For points on the set $D_{j_0}(f)$, the inequality \eqref{eqn:fboundinDj0} implies
		$$\int_{D_{j_0}(f)\cap \Omega} |f(x)|^p dx\leq C_2(a)^p(1+a)^{pj_0}\mu(\Omega)$$
		and $$\frac{1}{r}\sum_{\nu:\xi_\nu\in D_{j_0}(f)} |f(\xi_\nu)|^p\leq C_2(a)^p(1+a)^{pj_0}.$$
		By \eqref{eqn:fdiscreteinq} we have 
		\begin{align*}
		\|S(f,\xi)\|_p^p&\leq C_2(a)^p(1+a)^{pj_0}+C_2(a)^p\|S(h(f),\xi)\|_p^p\\
		&\leq C_2(a)^p(1+a)^{pj_0}+\frac{C_2(a)^p}{\mu(\Omega)}\|h(f)\|_{L^p(\Omega)}^p+C_2(a)^p\sigma\\
		&\leq C_2(a)^p\Big(\frac{C_1(a)^{-p}}{\mu(\Omega)}\|f\|_{L^p(\Omega)}^p+(1+a)^{pj_0}+\sigma\Big).
		\end{align*}
		On the other hand, we have
		\begin{align*}
		\|S(f,\xi)\|_p^p&\geq C_1(a)^p\|S(h(f),\xi)\|_p^p\\
		&\geq C_1(a)^p\Big( \frac{1}{\mu(\Omega)}\|h(f)\|_{L^p(\Omega)}^p-\sigma \Big)\\
		&\geq C_1(a)^p\Big( \frac{C_2(a)^{-p}}{\mu(\Omega)}\int_{\Omega\setminus D_{j_0}(f)} |f(x)|^pdx -\sigma \Big)\\
		&\geq C_1(a)^p\Big( \frac{C_2(a)^{-p}}{\mu(\Omega)}\|f\|_{L^p(\Omega)}^p-\frac{C_2(a)^{-p}}{\mu(\Omega)}\int_{D_{j_0}(f)\cap \Omega} |f(x)|^pdx -\sigma \Big)\\
		&\geq C_1(a)^p\Big( \frac{C_2(a)^{-p}}{\mu(\Omega)}\|f\|_{L^p(\Omega)}^p-(1+a)^{pj_0}-\sigma \Big).
		\end{align*}
	\end{proof}

	The lemma implies that the discretization of the integral norm of $f\in V_M$ and corresponding simple function $h(f)$ are related. We aim to find the probability for which the random sample set $\xi$ satisfy the condition (\ref{eqn:assumeh}) on $h(f)$ for all $f\in V_{M,\Omega}.$  
	
	\section{Random Sampling}
	\label{sec:randomsampling}
	
	In this section we discuss the main result of this paper. In order to derive the probabilistic estimates, we make use of the following lemma \cite[Lemma 2.1]{bourgain1989approximation}.
	\begin{lemma}
		\label{lemma:probabilitybound}
		 Let $\{ g_\nu \}_{\nu=1}^r$ be independent random variables with zero mean on probability space $(X,\rho)$ such that $$\|g_\nu\|_{L^1(X,\rho)}\leq 2, ~~ \|g_\nu\|_{L^{\infty}(X,\rho)}\leq L, ~~~1\leq \nu\leq r.$$ Then for any $\eta\in (0,1)$ we have the following probability bound $$\PP\left\{ \Big|\sum_{\nu=1}^r g_\nu\Big|\geq r\eta \right\}<2\exp\Big( -\frac{r\eta^2}{8L} \Big).$$
	\end{lemma}
 	
 	It is easy to verify that the above lemma implies the following result.
	\begin{cor}\cite[Corollary 2.2]{bourgain1989approximation}
		\label{cor:samprob}
		Let $\xi=\{ \xi_\nu \}_{\nu=1}^r$ be a random points drawn from probability space $(X,\rho)$ and $\{ \ff_j \}_{j\in G}$ be a finite collection of finite set of functions in $L^1(X,\rho)$. Assume that for each $j\in G$ and $f\in \ff_j$, we have $$\|f\|_{L^1(X,\rho)}\leq 1, ~~ \|f\|_{L^{\infty}(X,\rho)}\leq L_j.$$ Then for each $j\in G$, for any $\eta_j \in (0,1)$ and for all $f\in \ff_j,$ we have $$\left| \|f\|_{L^1(X,\rho)}-\frac{1}{r}\sum_{\nu=1}^r |f(\xi_\nu)| \right|\leq \eta_j,$$ with probability at least $1-2\suml_{j\in G} |\ff_j|\exp\Big( -\frac{r\eta_j^2}{8L_j} \Big).$
	\end{cor}
	
	Let $p\in [1,\infty)$. For $j>j_0$, define  $$\ff_j:=\Big\{ \frac{4}{5}(1+a)^{pj}\chi_{D_j(f)}: f\in V_{M,\Omega} \Big\}.$$ By definition of $D_j(f)$, we consider those $j>j_0$ such that $C_1(a)(1+a)^{j}\leq D\mu(\Omega)^{\frac{1}{p}}$. Otherwise $D_j(f)$ is empty. Choose $J\in \ZZ$ such that $	C_1(a)(1+a)^{J}\leq D\mu(\Omega)^{\frac{1}{p}},$ i.e., $ \displaystyle 	
		J\leq \frac{\log (D\mu(\Omega)^{\frac{1}{p}}/C_1(a))}{\log(1+a)}.$

	Now, consider the index set $G=[j_0,J]\cap \ZZ.$ Then, we apply Corollary \ref{cor:samprob} for the collection of sets $\{\ff_j\}_{j\in G}$ and prove the following theorem.
	\begin{thm}
		\label{thm:saminqVN}
		Let $\xi=\{ \xi_\nu \}_{\nu=1}^r$ be a sequence of i.i.d. random points that are uniformly distributed over the compact set $\Omega.$ If the sample size satisfy $$r\geq \frac{10}{\sigma^2}d_M|G|^2,$$ then for every $f\in V_M$ the following inequality
		\begin{multline*}
		\frac{C_1(a)^p}{\mu(\Omega)}\Big(C_2(a)^{-p}\|f\|_{L^p(\Omega)}^p-(1+a)^{pj_0}\|f\|_{L^p(\RR^n)}^p-\sigma\|f\|_{L^p(\RR^n)}^p\Big)\leq \|S(f,\xi)\|_p^p\\
		\leq \frac{C_2(a)^p}{\mu(\Omega)}\Big(C_1(a)^{-p}\|f\|_{L^p(\Omega)}^p+(1+a)^{pj_0}\|f\|_{L^p(\RR^n)}^p+\sigma\|f\|_{L^p(\RR^n)}^p\Big)
		\end{multline*}
		holds with probability at least $1-2A_1|G|\exp\Big( -(\frac{r\sigma^2}{|G|^2}-d_M)(1+a)^{-pj_0} \Big)$, where $A_1=\oo(\exp(d_M\log\mu(\Omega)))$.
	\end{thm}
	
	\begin{proof}
		Let $\xi=\{ \xi_\nu \}_{\nu=1}^r$ be the random points uniformly, identically and independently distributed over $\Omega$ with probability measure $d\rho=\frac{1}{\mu(\Omega)}dx.$ Hence for each $j\in G$ and $\phi_j=\frac{4}{5}(1+a)^{pj}\chi_{D_j(f)}$, we have
		\begin{align*}
		\|\phi_j\|_{L^{\infty}(\Omega,\rho)}\leq \frac{4}{5}(1+a)^{pj}:=L_j,
		\end{align*}
		and
		\begin{align*}
		\|\phi_j\|_{L^1(\Omega,\rho)}&=\frac{1}{\mu(\Omega)}\int_{\Omega}  \frac{4}{5}(1+a)^{pj}\chi_{D_j(f)}\,dx\\
		&=\frac{4}{5\mu(\Omega)}\|(1+a)^{pj}\chi_{D_j(f)}\|_{L^1(\Omega)}\\
		&=\frac{4C_1(a)^{-p}}{5\mu(\Omega)}\|f\|_{L^p(\Omega)}^p\\
		&\leq C_2(a)^{-p}\leq 1.
		\end{align*}		
		
		Denote $\eta_j=\frac{4\sigma}{5|G|},$ then from Corollary \ref{cor:samprob} we have
		\begin{gather}
		\left| \frac{1}{r}\sum_{\nu=1}^r |\phi_j(\xi_\nu)|-\frac{1}{\mu(\Omega)}\|\phi_j\|_{L^1(\Omega)} \right|\leq \eta_j \notag\\
		\label{eqn:saminqphi_j}
		\frac{1}{\mu(\Omega)}\|\phi_j\|_{L^1(\Omega)}-\eta_j\leq \frac{1}{r}\sum_{\nu=1}^r |\phi_j(\xi_\nu)|\leq \frac{1}{\mu(\Omega)}\|\phi_j\|_{L^1(\Omega)}+\eta_j,
		\end{gather}
		holds for each $j\in G$ and $\phi_j\in \ff_j$ with minimum probability of $1-2\suml_{j\in G} |\ff_j|\exp\Big( -\frac{r\eta_j^2}{8L_j} \Big).$\\
		
		Since  the sets $D_{j}(f)$ are pairwise disjoint, we get
		\begin{align*}
		\sum_{j\in G} \frac{1}{r}\sum_{\nu=1}^r |\phi_j(\xi_\nu)|&=\frac{4}{5}\frac{1}{r}\sum_{\nu=1}^r\sum_{j\in G} (1+a)^{pj}\chi_{D_j(f)}(\xi_\nu)\\
		&=\frac{4}{5}\|S(h(f),\xi)\|_p^p,
		\end{align*}
		and
		\begin{align*}
		\sum_{j\in G} \|\phi_j\|_{L^1(C_R)}&=\frac{4}{5}\int_{\Omega}\sum_{j\in G}(1+a)^{pj}\chi_{D_j(f)}(x)\, dx\\
		&=\frac{4}{5}\|h(f)\|_{L^p(\Omega)}^p.
		\end{align*}
		
		Hence, for all $f\in V_{M,\Omega}$, \eqref{eqn:saminqphi_j} implies
		\begin{gather*}
		\frac{1}{\mu(\Omega)}\sum_{j\in G} \|\phi_j\|_{L^1(\Omega)}-\sum_{j\in G} \eta_j\leq \sum_{j\in G} \frac{1}{r}\sum_{\nu=1}^r |\phi_j(\xi_\nu)|\leq \frac{1}{\mu(\Omega)}\sum_{j\in G} \|\phi_j\|_{L^1(\Omega)}+\sum_{j\in G} \eta_j\\
		\frac{4}{5\mu(\Omega)}\|h(f)\|_{L^p(\Omega)}^p-\frac{4\sigma}{5}\leq \frac{4}{5}\|S(h(f),\xi)\|_p^p\leq \frac{4}{5\mu(\Omega)}\|h(f)\|_{L^p(\Omega)}^p+\frac{4\sigma}{5}\\
		\frac{1}{\mu(\Omega)}\|h(f)\|_{L^p(\Omega)}^p-\sigma\leq \|S(h(f),\xi)\|_p^p\leq \frac{1}{\mu(\Omega)}\|h(f)\|_{L^p(\Omega)}^p+\sigma.
		\end{gather*}
		Therefore, Lemma \ref{lemma:saminqcon} implies that \eqref{eqn:saminqVNp} hold for every $f\in V_{M,\Omega}$ with probability at least $1-2\suml_{j\in G} |\ff_j|\exp\Big( -\frac{r\eta_j^2}{8L_j} \Big).$\\
		
		In order to calculate the bound of $\suml_{j\in G} |\ff_j|\exp\Big( -\frac{r\eta_j^2}{8L_j} \Big),$ we first evaluate the bound of $|\ff_j|.$ By the definition of $D_j(f)$ and construction of $\ff_j$ we conclude that $$|\ff_j|\leq |\aa_j|.$$
		Therefore, from \eqref{eqn:coveringbound} we have
		\begin{align*}
		\log|\ff_j|&\leq \log|\aa_j|\leq d_M\log\Big( \frac{4D\mu(\Omega)^{\frac{1}{p}}}{a(1+a)^j} \Big)\\
		&\leq d_M\log\Big( \frac{4D\mu(\Omega)^{\frac{1}{p}}}{a} \Big)+ d_M(1+a)^{-pj}\\
		&\leq \frac{d_M\log(C_4\mu(\Omega))}{p}+d_M(1+a)^{-pj},
		\end{align*}
		where $C_4=(\frac{4D}{a})^p,$ and $A_1=\exp(p^{-1}d_M\log(C_4\mu(\Omega))).$\\
		
		Now,
		\begin{align*}
		\suml_{j\in G} |\ff_j|\exp\Big( -\frac{r\eta_j^2}{8L_j} \Big)&\leq \suml_{j\in G} A_1\exp\Big( d_M(1+a)^{-pj}-\frac{r\eta_j^2}{8L_j} \Big)\\
		&=\suml_{j\in G} A_1\exp\Big( d_M(1+a)^{-pj}-\frac{r\sigma^2}{10|G|^2}(1+a)^{-pj} \Big)\\
		&=\suml_{j\in G} A_1\exp\Big( -R(1+a)^{-pj} \Big)\\
		&\leq A_1|G|\exp\Big( -R(1+a)^{-pj_0} \Big),
		\end{align*}
		where $R=\frac{r\sigma^2}{10|G|^2}-d_M,$ for sufficiently large sample size $r$, we can chose $R>0.$
		
		Let $f\in V_M\setminus \{0\}$   be arbitrary. Then $g=\frac{f\mu(\Omega)^{\frac{1}{p}}}{\|f\|_{L^p(\RR^n)}}\in V_{M,\Omega}$ and satisfy \eqref{eqn:saminqVNp}. Therefore,
		\begin{multline*}
		\frac{C_1(a)^p}{\mu(\Omega)}\Big(C_2(a)^{-p}\|f\|_{L^p(\Omega)}^p-(1+a)^{pj_0}\|f\|_{L^p(\RR^n)}^p-\sigma\|f\|_{L^p(\RR^n)}^p\Big)\leq \|S(f,\xi)\|_p^p\\
		\leq \frac{C_2(a)^p}{\mu(\Omega)}\Big(C_1(a)^{-p}\|f\|_{L^p(\Omega)}^p+(1+a)^{pj_0}\|f\|_{L^p(\RR^n)}^p+\sigma\|f\|_{L^p(\RR^n)}^p\Big),
		\end{multline*}
		holds with probability at least $1-2A_1|G|\exp\Big( -R(1+a)^{-pj_0} \Big)$ with an additional assumption $R>0.$	
	\end{proof}
	
	\begin{thm}
		Let $\xi=\{ \xi_\nu: \nu=1,2,\dots r \}$ be a sequence of independent random variable that are drawn uniformly from $\Omega$. Suppose that $0<\tau<1$ is small enough and if the number of sample size $r$ satisfies $$r\geq \frac{10}{\sigma^2}d_M|G|^2=\oo(\mu(\Omega)(\log \mu(\Omega))^2),$$ then for every $f\in V^*(\Omega,\delta)$ the sampling inequality
		\begin{equation}
		\label{eqn:mainsaminq}
		\frac{r}{\mu(\Omega)}\Big(\frac{1}{5}-\frac{p2^{p+1}\tau}{5}-pD^{p-1}\tau\Big)\|f\|_{L^p(\RR^n)}^p\leq \sum_{\nu=1}^r |f(\xi_\nu)|^p\leq \frac{r}{\mu(\Omega)}\Big( \frac{2B}{A}+pD^{p-1} \Big)\|f\|_{L^p(\RR^n)}^p,
		\end{equation}
		holds with probability at least $1-2A_1|G|\exp\Big( -R(1+a)^{-pj_0} \Big).$
	\end{thm}
	
	\begin{proof}
		Without loss of generality, let $f\in V^*(\Omega,\delta)$ with $\|f\|_{L^p(\RR^n)}=1.$
		
		From Lemma \ref{lemma:p-approx}, for $\tau>0$ there exist compact set $M$ and $\tilde{f}\in V_M$ such that
		\begin{gather*}
		\|f-\tilde{f}\|_{L^{\infty}(\Omega)}<\frac{\tau}{\mu(\Omega)},\\
		\|f-\tilde{f}\|_{L^{p}(\Omega)}<\tau,
		\end{gather*}
		and
		\begin{align*}
		\tilde{f}(x)&=\suml_{\gamma\in \Gamma\cap M} c_{\gamma}F_{\gamma}(x),\\
		|\tilde{f}(x)|&\leq \Big( \sum_{\gamma\in \Gamma\cap M} |c_{\gamma}|^p \Big)^{\frac{1}{p}}\Big( \sum_{\gamma\in \Gamma\cap M} |F_{\gamma}(x)|^{p'} \Big)^{\frac{1}{p'}}\\
		&\leq B^{\frac{1}{p}}\|f\|_{L^p(\RR^n)}\Big( \sum_{\gamma\in \Gamma} |\Theta(x-\gamma)|^{p'} \Big)^{\frac{1}{p'}}\\
		&\leq D\|f\|_{L^p(\RR^n)}.
		\end{align*}
		Next, we get
		\begin{align*}
		\Big| \|f\|_{L^{p}(\Omega)}^p-\|\tilde{f}\|_{L^{p}(\Omega)}^p \Big|&\leq p(1+\tau)^{p-1}\tau\\
		&\leq p2^{p-1}\tau 
		\end{align*}
		\begin{equation}
		\label{eqn:mainCRnorminq}
		\|f\|_{L^{p}(\Omega)}^p-p2^{p-1}\tau \leq \|\tilde{f}\|_{L^{p}(\Omega)}^p\leq \|f\|_{L^{p}(\Omega)}^p+p2^{p-1}\tau,
		\end{equation}
		and
		\begin{align*}
		\Big| |f(\xi_\nu)|^p-|\tilde{f}(\xi_\nu)|^p \Big|&\leq p\Big( \max\{ |f(\xi_\nu)|,|\tilde{f}(\xi_\nu)| \} \Big)^{p-1}|f(\xi_\nu)-\tilde{f}(\xi_\nu)|\\
		&\leq \frac{pD^{p-1}\tau}{\mu(\Omega)}.
		\end{align*}
		Therefore,
		\begin{align}
		\label{eqn:mainsuminq}
		\sum_{\nu=1}^r |\tilde{f}(\xi_\nu)|^p-\frac{r pD^{p-1}\tau}{\mu(\Omega)}\leq \sum_{\nu=1}^r |f(\xi_\nu)|^p\leq \sum_{\nu=1}^r |\tilde{f}(\xi_\nu)|^p+\frac{rpD^{p-1}\tau}{\mu(\Omega)}.
		\end{align}
		Since $V_M\subseteq V$ and using \eqref{eqn:frameinq}, we get $$\|\tilde{f}\|_{L^p(\RR^n)}^p\leq \frac{1}{A}\sum_{\gamma\in \Gamma\cap M} |c_{\gamma}|^p\leq \frac{1}{A}\sum_{\gamma\in \Gamma} |c_{\gamma}|^p\leq \frac{B}{A}\|f\|_{L^p(\RR^n)}^p.$$
		
		Hence, it follows from Theorem \ref{thm:saminqVN}, and equations \eqref{eqn:mainsuminq} and \eqref{eqn:mainCRnorminq} that
		\begin{align}
		\label{eqn:ransam1}
		\frac{r}{\mu(\Omega)}C_1(a)^p\Big(C_2(a)^{-p}\|&\tilde{f}\|_{L^p(\Omega)}^p-(1+a)^{pj_0}\|\tilde{f}\|_{L^p(\RR^n)}^p-\sigma\|\tilde{f}\|_{L^p(\RR^n)}^p\Big)\leq \sum_{\nu=1}^r |\tilde{f}(\xi_\nu)|^p \nonumber\\
		&\leq \frac{r}{\mu(\Omega)}C_2(a)^p\Big(C_1(a)^{-p}\|\tilde{f}\|_{L^p(\Omega)}^p+(1+a)^{pj_0}\|\tilde{f}\|_{L^p(\RR^n)}^p+\sigma\|\tilde{f}\|_{L^p(\RR^n)}^p\Big) \nonumber\\
		\frac{r}{\mu(\Omega)}C_1(a)^p\Big(C_2(a)^{-p}\|&f\|_{L^p(\Omega)}^p-C_2(a)^{-p}p2^{p-1}\tau-(1+a)^{pj_0}\frac{B}{A}-\sigma\frac{B}{A}\Big) \nonumber\\
		&\leq \sum_{\nu=1}^r |\tilde{f}(\xi_\nu)|^p\leq \frac{r}{\mu(\Omega)}C_2(a)^p\Big(C_1(a)^{-p}+(1+a)^{pj_0}+\sigma\Big)\frac{B}{A} \nonumber\\
		\frac{r}{\mu(\Omega)}\Big[C_1(a)^p\Big(C_2(a)^{-p}(1&-\delta)-(1+a)^{pj_0}\frac{B}{A}-\sigma\frac{B}{A}\Big)-\frac{4}{5}p2^{p-1}\tau-pD^{p-1}\tau\Big] \nonumber\\
		&\leq \sum_{\nu=1}^r |f(\xi_\nu)|^p \leq \frac{r}{\mu(\Omega)}\Big[ C_2(a)^p\Big(C_1(a)^{-p}+(1+a)^{pj_0}+\sigma\Big)\frac{B}{A}+pD^{p-1} \Big].
		\end{align}
		Now, we choose $j_0$ and $\sigma$ such that
		\begin{gather*}
		\frac{\sigma B}{A}=\frac{C_2(a)^{-p}(1-\delta)}{2},\\
		\frac{B(1+a)^{pj_0}}{A}=\frac{C_2(a)^{-p}(1-\delta)}{4}.
		\end{gather*}
		This implies, $$|j_0|=\log\Big( \frac{4BC_2(a)^p}{A(1-\delta)} \Big)/p\log(1+a).$$
		and
		\begin{align*}
		|G|\leq J+|j_0|&\leq \log (D^pC_1(a)^{-p}\mu(\Omega))/p\log(1+a)+\log\Big( \frac{4BC_2(a)^p}{A(1-\delta)} \Big)/p\log(1+a)\\
		&\leq \frac{\log (C_5\mu(\Omega))}{p\log(1+a)},
		\end{align*}
		where $C_5=\frac{5BD^p}{A(1-\delta)}.$
		Also, $\frac{C_2(a)^p}{C_1(a)^p}\leq \frac{5}{4}.$\\
		
		Hence, bounds of the sampling inequality \eqref{eqn:ransam1} can be revised. We get the lower estimate as
		\begin{align*}
		\frac{r}{\mu(\Omega)}\Big(\frac{1}{5}-\frac{4}{5}p2^{p-1}\tau-pD^{p-1}\tau\Big)\leq \frac{r}{\mu(\Omega)}\Big(\frac{4C_1(a)^p}{5C_2(a)^p}-\frac{4}{5}p2^{p-1}\tau-pD^{p-1}\tau\Big),
		\end{align*}
		and upper bound as
		\begin{align*}
		\frac{r}{\mu(\Omega)}\Big( \frac{5B}{4A}+\frac{1-\delta}{4}+\frac{1-\delta}{2}+pD^{p-1} \Big)\leq \frac{r}{\mu(\Omega)}\Big( \frac{2B}{A}+pD^{p-1} \Big).
		\end{align*}
		This implies for every $f\in V^*(\Omega,\delta)$
		\begin{equation*}
		\frac{r}{\mu(\Omega)}\Big(\frac{1}{5}-\frac{p2^{p+1}\tau}{5}-pD^{p-1}\tau\Big)\|f\|_{L^p(\RR^n)}^p\leq \sum_{\nu=1}^r |f(\xi_\nu)|^p\leq \frac{r}{\mu(\Omega)}\Big( \frac{2B}{A}+pD^{p-1} \Big)\|f\|_{L^p(\RR^n)}^p,
		\end{equation*}
		holds with probability at least $1-2A_1|G|\exp\Big( -R(1+a)^{-pj_0} \Big).$
	\end{proof}
	
	\begin{rem}
		Let $\{ \xi_\nu \}$ be random sample drawn uniformly from $\Omega$. The sampling inequality \eqref{eqn:mainsaminq} hold with probability at least $1-\epsilon$ if
		\begin{align*}
		2A_1|G|\exp\Big( -R(1+a&)^{-pj_0} \Big)\\
		&\leq 2|G|\exp\Big( d_Mp^{-1}\log(C_4\mu(\Omega))-(1+a)^{-pj_0}\Big( \frac{r\sigma2}{10|G|^2}-d_M \Big) \Big)\\
		&\leq 2\frac{\log(C_5\mu(\Omega))}{p\log(1+a)}\exp\Big( C(\omega)\mu(\Omega)p^{-1}\log(C_4\mu(\Omega))\\
		&\hspace{2cm}-(1+a)^{-pj_0}\Big( \frac{r\sigma^2p^2(\log(1+a))^2}{(\log(C_5\mu(\Omega)))^2}-C(\omega)\mu(\Omega) \Big) \Big)\\
		&<\epsilon.
		\end{align*}
		Hence, if the sample size $r\geq \oo\big( \mu(\Omega)(\log \mu(\Omega))^3 \big)$, then sampling inequality \eqref{eqn:mainsaminq} holds with probability at least $1-\epsilon$.
	\end{rem}

	\section*{Acknowledgements}
	The first author is thankful to Council of Scientic \& Industrial Research for financial support. The second author acknowledges Department of Science and Technology, Government of India for the financial support through project no. CRG/2019/002412.
	
	\bibliographystyle{plain}
	\bibliography{paper2}

\begin{thebibliography}{10}

\bibitem{averbuch2001fast}
Amir Averbuch, RR~Coifman, DL~Donoho, Moshe Israeli, and Johan Walden.
\newblock Fast slant stack: A notion of radon transform for data in a cartesian
  grid which is rapidly computible, algebraically exact, geometrically faithful
  and invertible.
\newblock {\em SIAM Scientific Computing}, 37(3):192--206, 2001.

\bibitem{bass2005random}
Richard~F Bass and Karlheinz Gr{\"o}chenig.
\newblock Random sampling of multivariate trigonometric polynomials.
\newblock {\em SIAM Journal on Mathematical Analysis}, 36(3):773--795, 2005.

\bibitem{bass2010random}
Richard~F Bass and Karlheinz Gr{\"o}chenig.
\newblock Random sampling of bandlimited functions.
\newblock {\em Israel Journal of Mathematics}, 177(1):1--28, 2010.

\bibitem{bass2013relevant}
Richard~F Bass, Karlheinz Gr{\"o}chenig, et~al.
\newblock Relevant sampling of band-limited functions.
\newblock {\em Illinois Journal of Mathematics}, 57(1):43--58, 2013.

\bibitem{bourgain1989approximation}
Jean Bourgain, Joram Lindenstrauss, and Vitali Milman.
\newblock Approximation of zonoids by zonotopes.
\newblock {\em Acta mathematica}, 162:73--141, 1989.

\bibitem{candes2006robust}
Emmanuel~J Cand{\`e}s, Justin Romberg, and Terence Tao.
\newblock Robust uncertainty principles: Exact signal reconstruction from
  highly incomplete frequency information.
\newblock {\em IEEE Transactions on information theory}, 52(2):489--509, 2006.

\bibitem{christensen2003introduction}
Ole Christensen.
\newblock {\em An introduction to frames and Riesz bases}, volume~7.
\newblock Springer, 2003.

\bibitem{cucker2002mathematical}
Felipe Cucker and Steve Smale.
\newblock On the mathematical foundations of learning.
\newblock {\em Bulletin of the American mathematical society}, 39(1):1--49,
  2002.

\bibitem{cucker2007learning}
Felipe Cucker and Ding~Xuan Zhou.
\newblock {\em Learning theory: an approximation theory viewpoint}, volume~24.
\newblock Cambridge University Press, 2007.

\bibitem{dai2020sampling}
Feng Dai, A~Prymak, A~Shadrin, V~Temlyakov, and S~Tikhonov.
\newblock Sampling discretization of integral norms.
\newblock {\em Constructive Approximation}, pages 1--17, 2021.

\bibitem{delin1998pointwise}
Henrik Delin.
\newblock Pointwise estimates for the weighted {Bergman} projection kernel in
  {$\mathbf{C}^n$}, using a weighted {$L^2$} estimate for the
  {$\bar{\partial}$}.
\newblock In {\em Annales de l'institut Fourier}, volume~48, pages 967--997,
  1998.

\bibitem{dilmaghani2003novel}
Reza~Sham Dilmaghani, Mohammad Ghavami, Ben Allen, and Hamid Aghvami.
\newblock Novel {UWB} pulse shaping using prolate spheroidal wave functions.
\newblock In {\em 14th IEEE Proceedings on Personal, Indoor and Mobile Radio
  Communications, 2003. PIMRC 2003.}, volume~1, pages 602--606. IEEE, 2003.

\bibitem{donoho2006compressed}
David~L Donoho.
\newblock Compressed sensing.
\newblock {\em IEEE Transactions on information theory}, 52(4):1289--1306,
  2006.

\bibitem{feichtinger2008perturbation}
Hans~G Feichtinger, Ursula Molter, and Jos{\'e}~Luis Romero.
\newblock Perturbation techniques in irregular spline-type spaces.
\newblock {\em International Journal of Wavelets, Multiresolution and
  Information Processing}, 6(02):249--277, 2008.

\bibitem{fuhr2014relevant}
Hartmut Führ and Jun Xian.
\newblock Relevant sampling in finitely generated shift-invariant spaces.
\newblock {\em Journal of Approximation Theory}, 240:1--15, 2019.

\bibitem{frieden1971viii}
B~Roy Frieden.
\newblock {VIII} {Evaluation}, design and extrapolation methods for optical
  signals, based on use of the prolate functions.
\newblock In {\em Progress in optics}, volume~9, pages 311--407. Elsevier,
  1971.

\bibitem{grochenig2004localization}
Karlheinz Gr{\"o}chenig.
\newblock Localization of frames, {Banach} frames, and the invertibility of the
  frame operator.
\newblock {\em Journal of Fourier Analysis and Applications}, 10(2):105--132,
  2004.

\bibitem{grochenig2019strict}
Karlheinz Gr{\"o}chenig, Antti Haimi, Joaquim Ortega-Cerd{\`a}, and
  Jos{\'e}~Luis Romero.
\newblock Strict density inequalities for sampling and interpolation in
  weighted spaces of holomorphic functions.
\newblock {\em Journal of Functional Analysis}, 277(12):108282, 2019.

\bibitem{haykin1998signal}
Simon Haykin and David~J Thomson.
\newblock Signal detection in a nonstationary environment reformulated as an
  adaptive pattern classification problem.
\newblock {\em Proceedings of the IEEE}, 86(11):2325--2344, 1998.

\bibitem{hoogenboom2006localizing}
Nienke Hoogenboom, Jan-Mathijs Schoffelen, Robert Oostenveld, Laura~M Parkes,
  and Pascal Fries.
\newblock Localizing human visual gamma-band activity in frequency, time and
  space.
\newblock {\em Neuroimage}, 29(3):764--773, 2006.

\bibitem{kumar2020sampling}
Anuj Kumar and Sivananthan Sampath.
\newblock Sampling and average sampling in quasi shift-invariant spaces.
\newblock {\em Numerical Functional Analysis and Optimization},
  41(10):1246--1271, 2020.

\bibitem{laska2006random}
Jason Laska, Sami Kirolos, Yehia Massoud, Richard Baraniuk, Anna Gilbert, Mark
  Iwen, and Martin Strauss.
\newblock Random sampling for analog-to-information conversion of wideband
  signals.
\newblock In {\em 2006 IEEE Dallas/CAS Workshop on Design, Applications,
  Integration and Software}, pages 119--122. IEEE, 2006.

\bibitem{li2020random}
Yaxu Li, Qiyu Sun, and Jun Xian.
\newblock Random sampling and reconstruction of concentrated signals in a
  reproducing kernel space.
\newblock {\em Applied and Computational Harmonic Analysis}, 54:273--302, 2021.

\bibitem{lu2019non}
Yancheng Lu and Jun Xian.
\newblock Non-uniform random sampling and reconstruction in signal spaces with
  finite rate of innovation.
\newblock {\em Acta Applicandae Mathematicae}, pages 1--31, 2019.

\bibitem{nashed2010sampling}
M~Zuhair Nashed and Qiyu Sun.
\newblock Sampling and reconstruction of signals in a reproducing kernel
  subspace of {$L^p(\mathbb{R}^d)$}.
\newblock {\em Journal of Functional Analysis}, 258(7):2422--2452, 2010.

\bibitem{olevskii2016functions}
Alexander~M Olevskii and Alexander Ulanovskii.
\newblock {\em Functions with disconnected spectrum}, volume~65.
\newblock American Mathematical Soc., 2016.

\bibitem{patel2020random}
Dhiraj Patel and Sivananthan Sampath.
\newblock Random sampling in reproducing kernel subspaces of
  {$L^p(\mathbb{R}^n)$}.
\newblock {\em Journal of Mathematical Analysis and Applications},
  491(1):124270, 2020.

\bibitem{poggio2003mathematics}
Tomaso Poggio and Steve Smale.
\newblock The mathematics of learning: Dealing with data.
\newblock {\em Notices of the AMS}, 50(5):537--544, 2003.

\bibitem{rauth1998smooth}
Michael Rauth and Thomas Strohmer.
\newblock Smooth approximation of potential fields from noisy scattered data.
\newblock {\em Geophysics}, 63(1):85--94, 1998.

\bibitem{smale2005shannon}
Steve Smale and Ding-Xuan Zhou.
\newblock Shannon sampling {II}: Connections to learning theory.
\newblock {\em Applied and Computational Harmonic Analysis}, 19(3):285--302,
  2005.

\bibitem{strohmer1997computationally}
Thomas Strohmer.
\newblock Computationally attractive reconstruction of bandlimited images from
  irregular samples.
\newblock {\em IEEE Transactions on image processing}, 6(4):540--548, 1997.

\bibitem{strohmer1996recover}
Thomas Strohmer, Thomas Binder, and M~Sussner.
\newblock How to recover smooth object boundaries in noisy medical images.
\newblock In {\em Proceedings of 3rd IEEE International Conference on Image
  Processing}, volume~1, pages 331--334. IEEE, 1996.

\bibitem{yang2019random}
Jianbin Yang.
\newblock Random sampling and reconstruction in multiply generated
  shift-invariant spaces.
\newblock {\em Analysis and Applications}, 17(02):323--347, 2019.

\bibitem{tao2019random}
Jianbin Yang and Xinzhu Tao.
\newblock Random sampling and approximation of signals with bounded
  derivatives.
\newblock {\em Journal of Inequalities and Applications}, 2019(1):107, 2019.

\bibitem{yang2013random}
Jianbin Yang and Wei Wei.
\newblock Random sampling in shift invariant spaces.
\newblock {\em Journal of Mathematical Analysis and Applications},
  398(1):26--34, 2013.

\end{thebibliography}
\end{document}